\documentclass{amsart}

\usepackage{amsmath,amsfonts,amsthm,latexsym}
\usepackage{mathrsfs, textcomp}
\usepackage{tikz-cd}
\usetikzlibrary{arrows, matrix}
\def\ifif {if and only if\ \ }

\theoremstyle{definition}

\numberwithin{equation}{section}

\usepackage{amssymb}
\usepackage{amscd}
\usepackage{eqlist, color}
\usepackage{mathrsfs, textcomp}
\usepackage{color, eqlist}
\usepackage{graphicx}
\usepackage{bm}
\usepackage{amsmath,amsfonts,amsthm,latexsym}
\usepackage[all]{xy}
 
\DeclareMathOperator{\Max}{Max}

\theoremstyle{definition}\newtheorem{thm}{Theorem}[section]
\theoremstyle{definition}\newtheorem{cor}[thm]{Corollary}
\theoremstyle{definition}\newtheorem{lem}[thm]{Lemma}
\theoremstyle{definition}\newtheorem{prop}[thm]{Proposition}
\theoremstyle{definition}\newtheorem{defn}[thm]{Definition}
\theoremstyle{definition}\newtheorem{Rem}[thm]{Remark}
\theoremstyle{remark}
\theoremstyle{definition}\newtheorem{exam}[thm]{Example}
\theoremstyle{definition}

\begin{document}
\title[Homotopic subsets]
{Homotopic subsets of continuous maps and their applications}

\author{Ali Taherifar}
\address{Department of Mathematics, Yasouj University, Yasouj, Iran
\newline \indent}
\email{ataherifar@yu.ac.ir, ataherifar54@gmail.com}





\begin{abstract}
In this paper, we introduce the notion of bi-homotopy between subsets of continuous functions. A map $\phi$ from $A$ to $B$ is called an $h$-map if, for each two homotopic maps $f, g\in A$, their image (i.e., $\phi(f), \phi(g)$) are homotopic in $B$. We call an $h$-map $\phi$ from $A$ to $B$ a bi-homotopy if it satisfies two conditions. First, for any $f, g \in A$, $\phi(f)$ is homotopic to $\phi(g)$ in $B$ implies $f$ is homotopic to $g$ in $A$. Next, for each $g \in B$, there exists an $f \in A$ such that $\phi(f)$ is homotopic to $g$ in $B$.
We establish the concept of homotopy equivalence between subsets $A$ and $B$ (denoted as $A \simeq B$) as the existence of two bi-homotopies $\phi$ from $A$ to $B$ and $\psi$ from $B$ to $A$, satisfying $\phi\psi(h)$ is homotopic to $h$ for every $h \in B$, and $\psi\phi(h)$ is homotopic to $h$ for every $h \in A$.
We then apply this definition to characterize homotopic subsets of continuous functions and introduce novel categories of subsets of $C(X, Y)$, notably the category $\mathcal{P}(C(X, Y))$, where $X, Y$ are two topological spaces. In this category, objects represent subsets of $C(X, Y)$, morphisms denote bi-homotopies between any two objects, and a composition law governs the combination of morphisms.
Furthermore, we extend this framework to define homotopic groups (resp., rings) of continuous functions when $Y$ is a topological group (resp., topological ring). Leveraging topological properties of $X$ and $Y$, we investigate the group (resp., ring) properties of $C(X, Y)$. We discuss potential applications and implications of the introduced bi-homotopy concept in the study of continuous functions and their subsets.
\end{abstract}
 \maketitle


\section*{Introduction}\label{intro}
Let $X$ and $Y$ be two topological spaces, and let $C(X, Y)$ denotes the set of all continuous functions from $X$ to $Y$. When we discuss $C(X, Y)$ as a group (resp., ring), we assume that $Y$ is a topological group (resp., topological ring). Homotopy theory is an essential topic in algebraic topology, with numerous applications that can address various difficulties and challenges in topology. It plays a fundamental role in analyzing homotopy groups and their different types.

In this paper, we introduce the concept of bi-homotopy between two subsets of continuous functions. We observe that the concept of bi-homotopy offers new perspectives for understanding the topological and algebraic properties of subsets of continuous functions and their relationships with the underlying spaces. Utilizing this concept, we construct two new categories: the category $\mathcal{P}(C(X, Y))$ and the category $\pi_1[\mathcal{P}(C(X, Y))] = \{\pi_1[A]: A \subseteq C(X, Y)\}$. We prove that the map $\pi_1: \mathcal{P}(C(X, Y)) \longrightarrow \pi_1[\mathcal{P}(C(X, Y))]$ is a functor. Additionally, if $X$ and $Y$ are two homeomorphic spaces, then for any space $Z$, the categories $\mathcal{P}(C(X, Z))$ (resp., $\pi_1[\mathcal{P}(C(X, Y))]$) and $\mathcal{P}(C(Y, Z))$ (resp., $\pi_1[\mathcal{P}(C(Y, Z))]$) are equivalent (Theorem 2.3).

In Section 3 of this paper, we focus on $C(X, Y)$ as a group. First, we define the concept of group-homotopy (resp., $h$-homomorphism) between two groups of continuous functions. We observe that whenever $G$ is a group of continuous functions, $\pi_1[G] = \{[f]: f\in G\}$ is a quotient (factor) group of $G$. The categories $\mathbb{CGC}$ and $\pi_1[\mathbb{CGC}]$ are introduced. We conclude that the category $\pi_1[\mathbb{CGC}]$ is a subcategory of the category of groups, and the map $\pi_1: \mathbb{CGC}\longrightarrow\pi_1[\mathbb{CGC}]$ is a functor (Corollary 3.10).
Next, we prove that if $X$ is a topological space, and $Y$ and $Z$ are two topological groups, then the two groups $\pi_1[C(X, Y \times Z)]$ and $\pi_1[C(X, Z)] \times \pi_1[C(Y, Z)]$ are isomorphic (Theorem 3.11).

In Section 4, we define the concept of ring-homotopy between two rings of continuous functions. We further prove that if two rings of continuous functions, say $A$ and $B$, are homotopic, then two spaces $\Max_H(A)$ and $\Max_H(B)$ are homeomorphic (Theorem 4.4). We observe that when A is a ring of continuous maps, $\pi_1[A]=\{[f]: f\in A\}$ is a factor ring of $A$. Then, ideals (i.e., every ideal, prime ideal and maximal ideal) of $\pi_1[A]$ are characterized (Proposition 4.8).

The reader is referred  to the \cite{H, M} for topological concepts and for algebraic concepts see \cite{AM} and \cite{R_1}.
\section{Preliminaries}
We begin this section with the following definitions.
\begin{defn}
Given two subsets $A$ and $B$ of continuous functions, a map $\phi$ from $A$ to $B$ is called an $h$-map if for any two functions $f, g$ in $A$ that are homotopic ($f\simeq g$), the images $\phi(f)$ and $\phi(g)$ in $B$ are also homotopic. In other words, $\phi$ preserves homotopy between functions in $A$. 
\end{defn}
\begin{defn}
A bi-homotopy $\phi$ is a type of $h$-map that satisfies two additional properties. First, if $\phi(f)$ and $\phi(g)$ are homotopic in $B$, then $f$ and $g$ are homotopic in $A$. Second, for any function $g$ in $B$, there exists a function $f$ in $A$ such that $\phi(f)$ is homotopic to $g$ in $B$.
\end{defn}
If there exist two bi-homotopies $\phi$ from $A$ to $B$ and $\psi$ from $B$ to $A$ such that $\phi(\psi(h))$  is homotopic to $h$ for each $h$ in $B$ and $\psi(\phi(h))$ is homotopic to $h$ for each $h\in A$, then we say that $A$ and $B$ are homotopic (denoted by $A\simeq B$). This means that $A$ and $B$ have the same homotopy type.
\begin{lem}\label{0}
Two subsets $A, B$ of $C(X, Y)$ are homotopic \ifif there exists two $h$-maps $\phi: A\longrightarrow B$ and $\psi: B\longrightarrow A$ such that $\phi\psi(h)\simeq h$ for each $h\in B$ and
 $\psi\phi(h)\simeq h$ for each $h\in A$.
 \end{lem}
 \begin{proof}
 The necessity of the lemma is trivial. For sufficiency, suppose that there exist $h$-maps $\phi: A\longrightarrow B$ and $\psi: B\longrightarrow A$ such that $\phi\psi(h)\simeq h$ for each $h\in B$ and $\psi\phi(h)\simeq h$ for each $h\in A$. We claim that $\phi$ is a bi-homotopy, and the proof for $\psi$ is similar. Let $f, g\in A$ and $\phi(f)\simeq \phi(g)$. Then by hypothesis, we have $f\simeq\psi(\phi(f))\simeq\psi(\phi(g))\simeq g$. Consider $g\in B$. Then $\psi(g)\in B$. Put $f=\psi(g)$. Then we have $\phi(f)=\phi\psi(g))\simeq g$, by hypothesis. So we are done.
 \end{proof}
 \begin{exam}
 (1) For each two continuous functions $f, g$, if we put $A=\{f\}$ and $B=\{g\}$, then trivially $A$ and $B$ are homotopic.

 (2) Let $A=\{f\}$ and $B=\{g_{1}, g_{2}\}$, where $f, g_1, g_2\in C(X, Y)$. Then  $A\simeq B$ \ifif $g_1$ and $g_2$ are homotopic. To see it, first assume $\phi$ is a bi-homotopy from $A$ to $B$. Then by Condition 2 in the definition, $g_1\simeq\phi(f)\simeq g_2$. Next, assume $g_1$ and $g_2$ are homotopic. Then $\phi=\{(f, g_1)\}$ and $\psi=\{(g_1, f), (g_2, f)\}$ are two bi-homotopies from $A$ to $B$ and $B$ to $A$, respectively, and we have $\phi\psi(h)\simeq h$ for each $h\in B$ and $\psi\phi(h)\simeq h$ for each $h\in A$, respectively. Thus $A$ and $B$ are homotopic.

 (3) Consider two functions $f, g: S^1\longrightarrow S^1$ by $f(x)=x$ and $g(x)=-x$. Then trivially $f\simeq g$. Now, by Part (2),  $A=\{f\}$ and $B=\{f, g\}$  are homotopic subsets of $C(S^1, S^1)$.

 (4) Let $A=\{f_1, g_1\}$ and $B=\{f_2, g_2\}$ as two subsets of $C(X, Y)$. Then, we have $A\simeq B$ if and only if one of the following conditions holds.
 \begin{enumerate}
 \item $f_1\simeq g_1$ and $f_2\simeq g_2$;

\item $f_1$ is not homotopic to $g_1$ and $f_2$ is not homotopic to $g_2$.
 \end{enumerate}

 \end{exam}
  It's worth noting that we may have a bi-homotopy $\phi$ from $A$ to $B$ but $\phi^{-1}$ is not a bi-homotopy from $B$ to $A$. To see example, let $A=\{f_1, f_2\}$, where $f_1$ and $f_2$ are two non-homotopic functions, and $B=\{g_1, g_2\}$, where $g_{1}\simeq g_2$ and $A, B\subseteq C(X, Y)$ for two spaces $X, Y$. Then $\phi=\{(f_1, g_1), (f_2, g_2)\}$ is a bi-homotopy from $A$ to $B$. But $\phi^{-1}=\{(g_1, f_1), (g_2, f_2)\}$ is not a bi-homotopy from $B$ to $A$. This example also shows us there exists a bi-homotopy from $A$ to $B$ but $A$ and $B$ are not homotopic. 

 Similarly, a bijection between $A$ and $B$ need not be a bi-homotopy. For instance, let $A=\{f_1, f_2\}$ and $B=\{g_1, g_2\}$ be two subsets of $C(X, Y)$, where $f_1\simeq g_1$ but $f_2$ and $g_2$ are not homotopic. Then $\phi=\{(f_1, g_1), (f_2, g_2)\}$ is a bijection from $A$ to $B$ which is not a bi-homotopy.
 \begin{prop}
 If $\phi: A\longrightarrow B$ is a bi-homotopy which is a bijection too, then $A$ and $B$ are bi-homotopic.
 \end{prop}
 \begin{proof}
 By Lemma \ref{0}, it suffices to show that $\phi^{-1}: B\longrightarrow A$ is an $h$-map. To do so, let $g_1, g_2\in B$ with $g_1\simeq g_2$. Since $\phi$ is a bi-homotopy, we have $\phi(\phi^{-1}(g_1))\simeq\phi(\phi^{-1}(g_2))$. Therefore, $\phi^{-1}(g_1)\simeq\phi^{-1}(g_2)$, which shows that $\phi^{-1}$ preserves homotopy classes.
 \end{proof}
 \begin{prop}\label{ali}
 Let $X, Y$, and $Z$ be three topological spaces.
 \begin{enumerate}
 \item If $X$ and $Y$ are
 two equivalent spaces, then $C(X, Z)\simeq C(Y, Z)$.

 \item If $X$ and $Y$ are
 two equivalent spaces, then $C(Z, X)\simeq C(Z, Y)$.
 \end{enumerate}
\end{prop}
 \begin{proof}
 (1) Since $X$ and $Y$ are equivalent spaces, there exist continuous maps $f:X\longrightarrow Y$ and $g:Y\longrightarrow X$ such that $fg\simeq I_Y$ and $gf\simeq I_X$. Define $\phi:C(X,Z)\longrightarrow C(Y,Z)$ by $\phi(h)=hg$ for $h\in C(X,Z)$, and define $\psi:C(Y,Z)\longrightarrow C(X,Z)$ by $\psi(h)=hf$ for $h\in C(Y,Z)$.
To show that $\phi$ and $\psi$ are $h$-maps, let $h_1\simeq h_2$ in $C(X,Z)$. Then $h_1f\simeq h_2f$ since $f$ is continuous, which implies that $\phi(h_1)\simeq\phi(h_2)$. Similarly, we can show that $\psi$ is an $h$-map.
Now, for each $h\in C(X,Z)$, we have $\psi\phi(h)=\psi(hg)=(hg)f=h(gf)\simeq hI_X\simeq h$, and similarly, for each $h\in C(Y,Z)$, we have $\phi\psi(h)\simeq h$. Hence, by Lemma \ref{0}, $\phi$ and $\psi$ are bi-homotopy equivalences, and we have $C(X,Z)\simeq C(Y,Z)$.

(2) The proof is similar to the proof of Part (1).
 \end{proof}
 It follows directly from Proposition \ref{ali} by setting $Z=\mathbb{R}$ for Part (1) and $Z=X$ and $Z=Y$ for Parts (2) and (3), respectively.
  \begin{cor}
 The following statements hold.
 \begin{enumerate}
 \item If $X$ and $Y$ are two equivalent spaces, then $C(X)\simeq C(Y)$.

 \item If $X$ and $Y$ are two equivalent spaces, then $C(X, X)\simeq C(X, Y)$.

 \item If $X$ and $Y$ are two equivalent spaces, then  $C(Y, X)\simeq C(Y, Y)$.
 \end{enumerate}
 \end{cor}
  
 It is well-known that, if $X$ is a non-empty contractible space, then there is an $x_0\in X$ such that $X$ and $\{x_0\}$ are of the same homotopy. On the other hand, for any topological space $Y$,  $C(\{x_0\}, Y)=\{\{(x_0, y)\}: y\in Y\}$. We can therefore deduce from Part 1 of Proposition \ref{ali} the following result.
 \begin{cor}
 If $X$ is a contractible space, then $C(X, Y)\simeq \{\{(x_0, y)\}: y\in Y\}$, for some $x_0\in X$ .
 \end{cor}
 For a subset $A$ of $C(X, Y)$, we denote by $\pi_1[A]$ the set of all $[f]=\{g\in C(X, Y): g\simeq f\}$, where $f\in A$.
 \begin{prop}
 Let $A, B\subseteq C(X, Y)$.
 If there exists a bi-homotopy from $A$ to $B$, then $\pi_1[A]$ and $\pi_1[B]$ are two equivalent sets.
 \end{prop}
 \begin{proof}
Let $\phi$ be a bi-homotopy from $A$ to $B$. Define a map $\psi$ from $\pi_1[A]$ to $\pi_1[B]$ by sending the homotopy class  $[f]$ in $A$ to the homotopy class $[\phi(f)]$ in $B$.
To show that $\psi$ is one-to-one, suppose $[f_1]$ and $[f_2]$ are homotopy classes  in $A$, and that $\psi([f_1])=\psi([f_2])$. Then $[\phi(f_1)]=[\phi(f_2)]$, which means that $\phi(f_1)$ and $\phi(f_2)$ are homotopic in $B$. Since $\phi$ is a bi-homotopy, this implies that $f_1$ and $f_2$ are homotopic in $A$, and thus $[f_1]=[f_2]$. This shows that $\psi$ is one-to-one.
To show that $\psi$ is onto, let $[g]$ be a homotopy class  in $\pi_1[B]$. Then there exists an $f\in A$ such that $\phi(f)$ is homotopic to $g$ in $B$. Hence, $\psi([f])=[\phi(f)]=[g]$, which shows that $\psi$ is onto.
Therefore, $\psi$ is a one-to-one and onto map between $\pi_1[A]$ and $\pi_1[B]$, which implies that $\pi_1[A]$ and $\pi_1[B]$ are equivalent.
 \end{proof}
\section{Bi-homotopy categories}
We begin with the following lemma.

\begin{lem}\label{rang}
The composition of two bi-homotopies is a bi-homotopy.
\end{lem}
\begin{proof}
Let $\phi: A\longrightarrow
B$ and $\psi: B\longrightarrow C$ be two bi-homotopies.  We want to show that $\psi\circ\phi:A\longrightarrow C$ is also a bi-homotopy. To see this, let $f,g\in A$ be two homotopic maps, i.e., $f\simeq g$.  Since $\phi$ is a bi-homotopy, $\phi(f)\simeq\phi(g)$ in $B$. Since $\psi$ is also a bi-homotopy, it follows that $\psi(\phi(f))\simeq\psi(\phi(g))$ in $C$.  Now, assume $\psi(\phi(f))=(\psi\circ\phi)(f)\simeq(\psi\circ\phi)(g)=\psi(\phi(g))$ for some $f, g\in A$. Then $\phi(f)\simeq\phi(g)$, since $\psi$ is a bi-homotopy. Again, since $\phi$ is a bi-homotopy, $f\simeq g$. Finally, consider $g\in C$. Then, there is an $h\in B$ such that $\psi(h)\simeq g$. Hence, there is an $f\in A$ such that $\phi(f)\simeq h$. Thus, $(\psi\circ\phi)(f)\simeq\psi(h)\simeq g$.
\end{proof}
\begin{prop}
The relation $\simeq$ on $\mathcal{P}(C(X, Y))$ is an equivalence
relation.
\end{prop}
\begin{proof}Trivially, each $A\in \mathcal{P}(C(X, Y))$ is homotopic with itself
and $A\simeq B$ implies $B\simeq A$. Now assume $A\simeq B$ and
$B\simeq C$. Then there are bi-homotopies  $\phi_{1}: A\longrightarrow
B$, $\phi_{2}:B\longrightarrow A$ and $\psi_{1}: B\longrightarrow C$ and
$\psi_{2}:C\longrightarrow B$ such that $\phi_{2}\phi_{1}(h)\simeq h$
and $\psi_{2}\psi_{1}(h)\simeq h$ for $h\in A$ and $h\in B$,
respectively. Thus  $\psi_{1}\phi_{1}:A\longrightarrow C$ and
$\phi_{2}\psi_{2}:C\longrightarrow A$ are bi-homotopies, by Lemma \ref{rang}. Moreover, for each $h\in A$,
\[(\phi_{2}\psi_{2}(\psi_{1}\phi_{1})(h))=\phi_{2}(\psi_{2}(\psi_{1}\phi_{1}(h)))\simeq\phi_{2}(\phi_{1}(h))\simeq h.\]
Similarly, we can prove that
$(\psi_{1}\phi_{1})(\phi_{2}\psi_{2})(h)\simeq h$ for each $h\in C$.
Thus $A\simeq C$. This shows that $\simeq$ is an equivalence
relation on $\mathcal{P}(C(X, Y))$.
\end{proof}
The Lemma \ref{rang} helps us to construct a category.
Let $X, Y$ be two topological spaces. Then $\mathcal{P}(C(X, Y))$ is a category. In this category objects are subsets of $C(X, Y)$ and morphisms are bi-homotopies between any two objects, say, for two subsets $A, B$ of $C(X, Y)$ we denote it by $\mathcal{S}(A, B)$. An identity morphism $id_{A}\in\mathcal{S}(A, A)$, for each object $A\in\mathcal{P}(C(X, Y))$,  and a composition law for each triple of objects $A, B, C$, as follows;
\[\circ:\mathcal{S}(A, B)\times\mathcal{S}(B, C)\longrightarrow\mathcal{S}(A, C)\quad\text{by}\quad \circ(\psi, \phi)=\phi\circ\psi,\]\[\text{where}\quad \psi\in \mathcal{S}(A, B)\quad\text{and}\quad\phi\in\mathcal{S}(B, C).\] Evidently, composition is associate. For $\phi\in\mathcal{S}(A, B)$, $id_{B}\circ\phi=\phi$ and  $\phi\circ id_A=\phi$.

We will use the category $\mathcal{P}(C(X, Y))$ to produce another category that will be used later. Consider $\pi_1[\mathcal{P}(C(X, Y))]=\{\pi_1[A]: A\in\mathcal{P}(C(X, Y))\}$. Let $\phi: A\longrightarrow B\in Mor(\mathcal{P}(C(X, Y))$. Define $\pi_1(\phi): \pi_1[A]\longrightarrow\pi_1[B]$ by $\pi_1(\phi)([f])=[\phi(f)]$. First, we claim that $\pi_1(\phi)$ is a map for each morphism $\phi\in Mor(\mathcal{P}(C(X, Y))$. Let $[f_1]=[f_2]\in \pi_1[A]$. Then $f_1\simeq f_2$. Thus $\phi(f_1)\simeq\phi(f_2)$, since  $\phi$ is an $h$-map. This shows that $\pi_1(\phi)([f_1])=[\phi(f_1)]=[\phi(f_2)]=\pi_1(\phi)([f_2])$. Next, assume $\pi_1(\phi_1): \pi_1[A]\longrightarrow\pi_1[B]$ and $\pi_1(\psi): \pi_1[B]\longrightarrow\pi_1[C]$  are two maps. Then we claim that $\pi_1(\psi)\pi_1(\phi)=\pi_1(\psi\phi)$. This shows that the set $\{\pi_1(\phi): \phi \in Mor(\mathcal{P}(C(X, Y))\}$ is closed under composition. Proof of the claim: Let $[f]\in\pi_1[A]$. Then we have, \[\pi_1(\psi)\pi_1(\phi)([f])=\pi_1(\psi)(\pi_1(\phi)([f]))=\pi_1(\psi)([\phi(f)])=[\psi(\phi(f))]=\pi_1(\psi\phi)([f]).\] On the other hand,  for each three morphisms $\phi: A\longrightarrow B$, $\psi: B\longrightarrow C$ and $\eta: C\longrightarrow D$, we have $\eta( \psi\phi)=(\eta\psi)\phi$. Hence $\pi_1(\eta)(\pi_1(\psi)\pi_1(\phi))=\pi_1(\eta)(\pi_1(\psi\phi))=\pi_1(\eta(\psi\phi))=\pi_1((\eta\psi)\phi)=(\pi_1(\eta)\pi_1(\psi))\pi_1(\phi).$ Thus the composition is associate. For the identity morphism $id_B$, we have $\pi_1(id_B)=id_{\pi_1[B]}$. For, assume $[f]\in\pi_1[B]$. Then $\pi_1(id_B)([f])=[id_B(f)]=[f]=id_{\pi_1[B]}([f])$. This implies for each $\phi\in\mathcal{S}(A, B)$, $id_{\pi_1[B]}\pi_1(\phi)=\pi_1(id_B)\pi_1(\phi)=\pi_1(id_B\phi)=\pi_1(\phi)$. Similarly, we can see that $\pi_1(\phi)\pi_1(id_A)=\pi_1(\phi)$. Thus $\pi_1[\mathcal{P}(C(X, Y))]$ is a category in which objects are of the form $\pi_1[A]$,  and morphisms are $\pi_1(\phi)$, where $A\in \mathcal{P}(C(X, Y))$ and $\phi\in Mor(\mathcal{P}(C(X, Y))$. We will now apply these categories in the next result.
\begin{thm}
Let $X, Y$ be two topological spaces. Then the following statements hold.
\begin{enumerate}
\item The mappings $A\longrightarrow \pi_1[A]$ and $\phi\longrightarrow \pi_1(\phi)$ define a functor \[\pi_1: \mathcal{P}(C(X, Y))\longrightarrow\pi_1[\mathcal{P}(C(X, Y))].\]

\item If $X, Y$ are two homeomorphic spaces, then for each space $Z$ two categories $\mathcal{P}(C(Y, Z))$ and $\mathcal{P}(C(X, Z))$ are  isomorphic.

\item If $X, Y$ are two homeomorphic spaces, then for each space $Z$ two categories $\pi_1[\mathcal{P}(C(Y, Z))]$ and $\pi_1[\mathcal{P}(C(X, Z))]$ are isomorphic.
\end{enumerate}
\end{thm}
\begin{proof}
(1) This follows from the comments before the theorem.

(2) Let $f: X\longrightarrow Y$ be the homeomorphism. We denote by $A_f$, the set of all $hf=h\circ f$, where $h\in A$. Define, \[F:\mathcal{P}(C(Y, Z))\longrightarrow\mathcal{P}(C(X, Z))\quad \text{by}\quad F(A)=A_f.\] We claim that $F$ is a functor. To see it, consider a morphism $\phi\in\mathcal{S}(A, B)$, where $A, B\in\mathcal{P}(C(Y, Z))$. Firstly, we must show that $F(\phi): F(A)\longrightarrow F(B)$ by $F(\phi)(hf)=\phi(h)f$ is a morphism (i.e., a bi-homotopy). Let $h_1f=h_2f$. Then $h_1=h_1ff^{-1}=h_2ff^{-1}=h_2$. Thus $\phi(h_1)f=\phi(h_2)f$. Consider  two elements $g_1f, g_2f\in F(A)=A_f$ such that $g_1f\simeq g_2f$. Then we have $g_1\simeq (g_1f)f^{-1}\simeq(g_2f)f^{-1}\simeq g_2$. Thus $\phi(g_1)\simeq\phi(g_2)$. This shows that: \[F(\phi)(g_1f)=\phi(g_1)f\simeq\phi(g_2)f=F(\phi)(g_2f).\] Now assume $F(\phi)(h_1f)\simeq F(\phi)(h_2f)$, for some $h_1, h_2\in A$. Then $\phi(h_1)f\simeq \phi(h_2)f$. Hence $\phi(h_1)\simeq \phi(h_1)ff^{-1}\simeq \phi(h_2)ff^{-1}\simeq \phi(h_2)$. This implies $h_1\simeq h_2$ and hence $h_1f\simeq h_2f$. Let $hf\in F(B)=B_f$. Since $\phi$ is a bi-homotopy from $A$ to $B$, there exists $g_1\in A$ such that $\phi(g_1)\simeq h$. This shows $\phi(g_1)f\simeq hf$, i.e., $F(\phi)(g_1f)\simeq hf$. Trivially, for each $A\in\mathcal{P}(C(Y, Z))$, $F(id_A)=id_{F(A)}$. Now, let $\phi_1\in \mathcal{S}(A, B)$ and $\phi_2\in\mathcal{S}(B, C)$. Then for $hf\in F(A)$ we have,
\[F(\phi_2\phi_1)(hf)=(\phi_2\phi_1)(h)f=\phi_2(\phi_1(h))f=F(\phi_2)(\phi_1(h)f)=\]\[F(\phi_2)(F(\phi_1)(hf))=(F(\phi_2)F(\phi_1))(hf).\] Thus $F$ is a covariant functor.
Now for each morphism $\phi\in Mor(\mathcal{P}(C(X, Z)))$, and $A, B\in \mathcal{P}(C(X, Z))$, we define,
\[G: \mathcal{P}(C(X, Z))\longrightarrow \mathcal{P}(C(Y, Z))\quad\text{by}\quad G(B)=B_{f^{-1}}\quad\text{and},\]\[\quad G(\phi): G(A)\longrightarrow G(B)\quad\text{by}\quad G(\phi)(hf^{-1})=\phi(h)f^{-1}.\] Then, similar to the proof of $F$, we can see that $G$ is a functor too. Moreover, $FG=Id_{\mathcal{P}(C(X))}$ and $GF=Id_{\mathcal{P}(C(Y))}$, where $Id$ is the identity functor.

(3) By hypothesis, there exists a homeomorphism $f: X\longrightarrow Y$. For our purpose define as in the proof of Part 1, $A_f=\{hf: h\in A\}$, and $B_{f^-1}=\{hf^{-1}: h\in B\}$, where $A\in \mathcal{P}(C(Y, Z))$ and $B\in \mathcal{P}(C(X, Z))$, respectively. For $\pi_1(\phi)\in Mor(\pi_1[\mathcal{P}(C(Y, Z))])$ and $A, B\in \mathcal{P}(C(Y, Z))$, define \[F: \pi_1[\mathcal{P}(C(Y, Z))]\longrightarrow\pi_1[\mathcal{P}(C(X, Z))]\quad \text{by}\quad F(\pi_1[A])=\pi_1[A_f]\quad\text{and},\]\[F(\pi_1(\phi)): \pi_1[A_f]\longrightarrow\pi_1[B_f]\quad\text{by}\quad F(\pi_1(\phi))([hf])=[\phi(h)f].\]  For two  $A, B\in \mathcal{P}(C(Y, Z))$, it is easy to see that $\pi_1[A]=\pi_1[B]$ implies $F(A)=\pi_1[A_f]=\pi_1[B_f]=F(B)$. Similar to the argument of Part 1,  we can see that $F(\pi_1(\phi))$ is a map. It is enough to show that  $F(\pi_1(\phi))$ is a morphism. In fact, we claim that $F(\pi_1(\phi))=\pi_1(F(\phi))$, where $F(\phi): A_f\longrightarrow B_f$ is defined by $F(\phi)(hf)=\phi(h)f$, hence this is a morphism. To see our claim, consider $[hf]\in A_f$. Then, $F(\pi_1(\phi))([hf])=[\phi(h)f]$ and $\pi_1(F(\phi))([hf])=[F(\phi)(hf)]=[\phi(h)f]$. Thus $F(\pi_1(\phi))([hf])=\pi_1(F(\phi))([hf])$. Now for each morphism $\phi\in Mor(\mathcal{P}(C(X, Z)))$, and $A, B\in \mathcal{P}(C(X, Z))$, we define,
\[G: \pi_1[\mathcal{P}(C(X, Z))]\longrightarrow \pi_1[\mathcal{P}(C(Y, Z))]\quad\text{by}\quad G(\pi_1[B])=\pi_1[B_{f^{-1}}]\quad\text{and},\]\[\quad G(\pi_1(\phi)): G(A)\longrightarrow G(B)\quad\text{by}\quad G(\phi)(hf^{-1})=\phi(h)f^{-1}.\] Then, similar to the proof of $F$, we can see that $G$ is a functor too. Moreover, $FG=Id_{\pi_1[{\mathcal{P}(C(X, Z))}]}$ and $GF=Id_{\pi_1[{\mathcal{P}(C(Y, Z))}]}$, where $Id$ is the identity functor.
\end{proof}

\section{group homotopy}
In this section, we study groups of continuous maps, which are of the form $C(X, Y)$, where $Y$ is a topological group. Let $*$ be the operation on $Y$.
For $\alpha, \beta\in C(X, Y)$ and $x\in X$, we define $(\alpha *\beta)(x)=\alpha(x) *\beta(x)$. Then we have $\alpha*\beta\in C(X, Y)$ and $(C(X, Y), *)$ is a group. Since, we have $\alpha *\beta=hk$, where $k: X\longrightarrow Y\times Y$ is $k(x)=(\alpha(x), \beta(x))$ and $h: Y\times Y\longrightarrow Y$ is $h(y_1, y_2)=y_1* y_2$. 
\begin{defn} Let $(G_1, o_1)$ and $(G_2, o_2)$ be two groups of continuous maps. An $h$-map $\phi$ from group $G_1$ to $G_2$ is called $h$-homomorphism if for each $f_1, f_2\in G_1$ we have 
\begin{equation*}
\phi(f_1 o_1 f_2)\simeq\phi(f_1) o_2 \phi(f_2).
\end{equation*}
\end{defn}
\begin{defn}
A bi-homotopy $\phi$ from group $(G_1, o_1)$ to group $(G_2, o_2)$ which satisfies in $\phi(f_{1}o_1 f_{2})\simeq \phi(f_{1})o_2 \phi(f_{2})$  for each $f_{1}, f_{2}\in G_1$ is called a group-homotopy. If there exist two group-homotopies $\phi$ from $G_1$ to $G_2$ and $\psi$ from $G_2$ to $G_1$, such that $\psi\phi(h)\simeq h$ for each $h\in G_1$ and $\phi\psi(h)\simeq h$ for each $h\in G_2$, then we say $G_1$ and $G_2$ are two homotopic groups and we denote it by $G_1\simeq^{g} G_2$.
\end{defn}
\begin{exam}\label{ol}
If $X, Y$ are two equivalent spaces and $Z$ is a topological group, then $C(X, Z)$ and $C(Y, Z))$ are two homotopic groups. For, assume $f: X\longrightarrow Y$ and $g: Y\longrightarrow X$ are two continuous maps such that $fg\simeq I_Y$ and $gf\simeq I_X$. It is enough to define $\phi: C(Y, Z)\longrightarrow C(X, Z)$ by $\phi(\alpha)=\alpha f$ and $\psi: C(X, Z)\longrightarrow C(Y, Z)$ by $\psi(\beta)=\beta g$. Then, it is easy to see that $\phi$ and $\psi$ are two group homotopies, $\phi\psi(\alpha)\simeq \alpha$ for each $\alpha\in C(Y, Z)$ and $\psi\phi(\beta)\simeq\beta$ for each $\beta\in C(X, Z)$, respectively.  Thus, $C(X,Z)\simeq^g C(Y, Z)$.
\end{exam}
\begin{exam} 
Let $X$ be a contractible space and $Y$ be a topological group (e.g., $Y=S^1$). Then $X$ is equivalent to the one point subset of itself say $\{x_0\}$. Thus $C(X, Y)\simeq^g C(\{x_0\}, Y)=Y$, by Example \ref{ol}.
\end{exam}
\begin{lem}
The composition of two $h$-homomorphisms is an $h$-homomorphism.
\end{lem}
\begin{proof}
Let $(G_1, o_1), (G_2, o_2)$ and $(G_3, o_3)$ be three groups and $\phi_1: G_1\longrightarrow G_2$, $\phi_2: G_2\longrightarrow G_3$ be two $h$-homomorphisms. By Lemma \ref{rang}, $\phi_2\phi_1: G_1\longrightarrow G_3$ is an $h$-map. Now let $f_1, f_2\in G_1$. Then by the hypothesis, \[(\phi_2 \phi_1)(f_1\circ_1 f_2)=\phi_2(\phi_1(f_1\circ_1 f_2))\simeq \phi_2(\phi_1(f_1)\circ_2 \phi_1(f_2))\simeq \phi_2(\phi_1(f_1))\circ_3 \phi_2(\phi_1(f_1)).\] Thus $\phi_2\phi_1$ is an $h$-homomorphism.
\end{proof}
By using the above lemma, we can introduce the category of groups of continuous maps which is denoted by $\mathbb{CGC}$. In this category objects are groups of continuous maps and morphisms are $h$-homomorphisms between any two objects. Let $H(G_1, G_2)$ be the set of all $h$-homomorphisms between two groups $G_1$ and $G_2$. For each group $G$ of continuous maps, $id_G\in H(G, G)$, is the identity morphism. The composition law for each triple $G_1, G_2$ and $G_3$ is defined as follows; \[\circ: H(G_1, G_2)\times H(G_2, G_3)\longrightarrow H(G_1, G_3)\quad\text{by}\quad \circ(\psi, \phi)=\phi\circ\psi,\] where $\psi\in H(G_1, G_2)$ and $\phi\in H(G_2, G_3)$. Trivially composition is associated. Also, for $\phi\in H(G_1, G_2)$, $id_{G_2}\circ \phi=\phi$ and $\phi\circ id_{G_1}=\phi$. 
\begin{Rem}
If $f_1\simeq f_2$ and $g_1\simeq g_2$ are elements of the group $(C(X, Y), o)$ with $o$ is the operation on $Y$ and $Y$ is a topological group, then $f_1o g_1\simeq f_2 o g_2$. For, assume $F: X\times I\longrightarrow Y$ and $G: X\times I\longrightarrow Y$ are homotopies between $f_1, g_1$ and $f_2, g_2$, respectively. Then $H: X\times I\longrightarrow Y$ by $H(x, t)=F(x, t) o G(x, t)$ is a homotopy between $f_1 o g_1$ and $f_2 o g_2$. For continuity of $H$, it is enough to consider the map $H_1: X\times I\longrightarrow Y\times Y$ by $H_1(x, t)=(F(x, t), G(x, t))$. Then $H_1$ is continuous. On the other hand, the map $h: Y\times Y\longrightarrow Y$ by $h(x, y)=x o y$ is continuous. Hence $H=h\circ H_1$ is continuous.  We use this fact to show that,
if $(G, o)$ is a group of continuous maps, then $\pi_1[G]$ is a group too. In fact, it is enough to define $o^{\prime}$ on $\pi_1[G]$ by
$[f] o^{\prime}[g]=[fo g].$ Then, it is easy to see that $(\pi_1[G], o^{\prime})$ is a group. 
\end{Rem}
For each group $(G, o)$ of continuous maps, define $P_{G}: G\longrightarrow \pi_1[G]$ by $P_G(g)=[g]$. Then for each $g_1, g_2\in G$, we have $P_G(g_1 o g_2)=[g_1o g_2]=[g_1] o^{\prime} [g_2]=P_G(g_1) o^{\prime} P_G(g_2)$. This equality  shows that $P_G$ is a group homomorphism. On the other hand, $ker\phi=\{g\in G: g\simeq o\}$ is a normal subgroup of $G$. Thus we have $G/ker\phi\cong\pi_1[G]$. This shows that $\pi_1[G]$ is a quotient group of $G$.  We will use this map in the following result.
\begin{thm}\label{gol}
Let $(G_1, o_1), (G_2, o_2)$ be two groups of continuous maps and $\phi: G_1\longrightarrow G_2$ be a map. Then the following statements hold.
\begin{enumerate}
\item $\pi_1(\phi)$ is a group homomorphism \ifif $\phi$ is an $h$-homomorphism.

\item $\pi_1(\phi)$  is a group isomorphism \ifif $\phi$ is a group-homotopy.

\item Let $\phi$ be an $h$-map which is a group homomorphism. Then the following diagram of group homomorphisms commutes.

\[\begin{tikzcd}
G_1 \arrow[r, red, "\phi"] \arrow[d, blue, "P_{G_1}"]
& G_2\arrow[d, red, "P_{G_2}"] \\
\pi_1[G_1] \arrow[r, blue,  "\pi_1(\phi)" blue]
&  \pi_1[G_2]
\end{tikzcd}\]
\end{enumerate}
\end{thm}
\begin{proof}
(1) $\Longrightarrow$ Let $f_1, f_2\in G_1$ and $f_1\simeq f_2$. Then $[f_1]=[f_2]$. Since $\pi_1(\phi)$ is a map, $[\phi(f_1)]=\pi_1(\phi)([f_1])=\pi_1(\phi)([f_2])=[\phi(f_2)]$. Thus $\phi(f_1)\simeq\phi(f_2)$. Thus $\phi$ is an $h$-map. To show $\phi$ is an $h$-homomorphism, consider $f_1, f_2\in G_1$. Then $[\phi(f_1 o_1 f_2)]=\pi_1(\phi)([f_1 o_1 f_2])=\pi_1(\phi)([f_1]) o^{\prime}_2\pi_1(\phi)([f_2])=[\phi(f_1)] o^{\prime}_2[\phi(f_2)]=[\phi(f_1) o_2\phi(f_2)]$, by the hypothesis. Thus $\phi(f_1 o_1 f_2)\simeq \phi(f_1) o_2\phi(f_2)$.

$\Leftarrow$ First, assume $[f_1], [f_2]\in \pi_1[G_1]$ and $[f_1]=[f_2]$. Then $f_1\simeq f_2$. By the hypothesis, $\phi(f_1)\simeq\phi(f_2)$, i.e., $\pi_1(\phi)([f_1])=[\phi(f_1)]=[\phi(f_2)]=\pi_1(\phi)([f_2])$. Next,  let $[g_1], [g_2]\in \pi_1[G_1]$. Then by the hypothesis we have,\[\pi_1(\phi)([g_1]o_1^{\prime} [g_2])=\pi_1(\phi)([g_1 o_1 g_2])=[\phi(g_1 o_1 g_2)]=\]\[[\phi(g_1) o_2 \phi(g_2)]=[\phi(g_1)]o_2^{\prime} [\phi(g_2)]=\pi_1(\phi)([g_1]) o_2^{\prime} \pi_1(\phi)([g_2]).\]

(2) $\Longrightarrow$ By Part (1), $\phi$ is an $h$-map.  Consider two elements $f_1, f_2\in G_1$ with $\phi(f_1)\simeq\phi(f_2)$. Then $\pi_1(\phi)([f_1])=[\phi(f_1)]=[\phi(f_2)]=\pi_1(\phi)([f_2]).$ By the hypothesis, $[f_1]=[f_2]$, i.e., $f_1\simeq f_2$. Now let $g\in G_2$. Then $[g]\in \pi_1[G_2]$. By the hypothesis, there exists $[h]\in\pi_1[G_1]$ (hence, $h\in G_1$) such that $\phi(h)=\pi_1(\phi)([h])=[g]$. This shows that $\phi(h)\simeq g$. Hence $\phi$ is a bi-homotopy. Finally, choose $f_1, f_2\in G_1$. By the hypothesis, we have \[[\phi(f_1 o_1 f_2)]=\pi_1(\phi)(f_1o_1 f_2)=\pi_1(\phi)(f_1)o_2^{\prime}\pi_1(\phi)(f_2)=[\phi(f_1)]o_2^{\prime}[\phi(f_2)]=[\phi(f_1)o_2\phi(f_2)].\] This implies $\phi(f_1o_1 f_2)\simeq \phi(f_1)o_2 \phi(f_2)$.

$\Leftarrow$ Let  $[h_1]=[h_2]$, where $h_1, h_2\in G_1$. Then we have $h_1\simeq h_2$. Since $\phi$ is a bi-homotopy, $\phi(h_1)\simeq \phi(h_2)$. Thus $\pi_1(\phi)([h_1])=[\phi(h_1)]=[\phi(h_2)]=\pi_1(\phi)([h_2])$. Now let $\pi_1(\phi)([f_1])=\pi_1(\phi)([f_2])$, where $f_1, f_2\in G_1$. Then $[\phi(f_1)]=[\phi(f_2)]$. This implies $\phi(f_1)\simeq\phi(f_2)$. Hence $f_1\simeq f_2$. Thus $[f_1]=[f_2]$. This shows $\pi_1(\phi)$ is one-one. Let $[g]\in \pi_1[G_2]$. Then $g\in G_2$. Thus,  there exists an $f\in A$ such that $\phi(f)\simeq g$. Thus $\pi_1(\phi)(f)=[\phi(f)]=[g]$. This says $\pi_1(\phi)$ is an onto map. It is enough to show that $\pi_1(\phi)$ is a group homomorphism. To see it, assume $[f_1], [f_2]\in \pi_1[G_1]$.  By the hypothesis, $\phi(f_1o_1f_2)\simeq \phi(f_1)o_2\phi(f_2)$. Thus we have \[\pi_1(\phi)([f_1]o_1^{\prime}[f_2])=\pi_1(\phi)([f_1o_1 f_2])=[\phi(f_1o_1 f_2)]=[\phi(f_1)o_2\phi(f_2)]=\]\[[\phi(f_1)]o_2^{\prime}[\phi(f_2)]=\pi_1(\phi)([f_1])o_2^{\prime}\pi_1(\phi)([f_2]).\]
(3) Trivial.
\end{proof}
Part 2 of the above Theorem implies the next result.
\begin{cor}
If $G_1, G_2$ are two homotopic groups of continuous functions, then $\pi_1[G_1]$ and $\pi_1[G_2]$ are two  isomorphic groups.
\end{cor}
Whenever $X, Y$ are two equivalent spaces and $Z$ is a topological group, $C(X, Z)\simeq^g C(Y, Z)$, by Example \ref{ol}. By using the above result, we obtain the following result.
\begin{cor}
If $X, Y$ are two equivalent spaces and $Z$ is a topological group, then $\pi_1(C(X, Z))$ and $\pi_1(C(Y, Z))$ are two isomorphic groups.
\end{cor} 
Let $\pi_1[\mathbb{CGC}] = \{\pi_1[G]: G$ is a group of continuous functions$\}$. Theorem \ref{gol}, along with the fact that for each $\phi, \psi \in \text{Mor}(\mathbb{CGC})$, we have $\pi_1(\phi\psi) = \pi_1(\phi)\pi_1(\psi)$, imply the following result.
\begin{cor}
The following statements hold.
\begin{enumerate}
\item  $\pi_1[\mathbb{CGC}]$ is a subcategory of the category of groups.

\item The mappings $G\longrightarrow \pi_1[G]$ and $\phi\longrightarrow\pi_1(\phi)$ define a functor 
\begin{equation*}
\pi_1: \mathbb{CGC}\longrightarrow\pi_1[\mathbb{CGC}].
\end{equation*}
\end{enumerate}
\end{cor}
Let $f: X\longrightarrow Y$ and $g: X\longrightarrow Z$ be two continuous maps. Then we have the continuous  map $((f, g)):X\longrightarrow Y\times Z$ by $((f, g))(x)=(f(x), g(x))$. 
\begin{lem}\label{tiam}
Let $f_1, g_1\in C(X, Y)$ and $f_2, g_2\in C(X, Z)$. Then $f_1\simeq g_1$ and $f_2\simeq g_2$ \ifif $((f_1, f_2))\simeq ((g_1, g_2)): X\longrightarrow Y\times Z$.
\end{lem}
\begin{proof}
Let $F: X\times I\longrightarrow Y$ be a homotopy between $f_1, g_1$ and $G: X\times I\longrightarrow Z$ be a homotopy between $f_2, g_2$. Define $H: X\times I\longrightarrow Y\times Z$ by $H(x, t)=(F(x, t), G(z, t))$. Then, $H$ is a homotopy between $((f_1, f_2))$ and $((g_1, g_2))$. 

Conversely, let $H:X\times I\longrightarrow Y\times Z$ be a homotopy between $((f_1, f_2))$ and $((g_1, g_2))$. Then $P_1H$ and $P_2H$ ($P_1, P_2$ are projection maps) are two homotopies between $f_1, g_1$ and $f_2, g_2$, respectively. 
\end{proof}
Now, we apply Lemma \ref{tiam} to find an example of an $h$-homomorphism which is not a group homotopy.
\begin{exam}
Define $\phi: C(S^1, S^1\times S^1)\longrightarrow C(S^1, S^1)$ by $\phi(f)=P_1f$, where $P_1$ is the projection map. Then $f\simeq g$ in $C(S^1, S^1\times S^1)$ implies $P_1f\simeq P_1g$, i.e., $\phi(f)\simeq\phi(g)$. Thus $\phi$ is an $h$-map. For $f, g\in C(S^1, S^1\times S^1)$, it is easy to see that $P_1(f*_1g)=P_1f* P_1g$, where $*_1$ is the operation on $S^1\times S^1$ and $*$ is the operation on $S^1$. This shows that $\phi$ is an $h$-homomorphism. However, this map is not a group homotopy (in fact, it is not a bi-homotopy). To see it, consider two continuous maps $f, g: S^1\longrightarrow S^1$ such that $f\simeq g$. Since, the identity map (i.e., $I_{S^{1}}$) on $S^1$ is not homotopic to the constant map $c$ on it. Thus two maps $f_1=((f, I_{S^{1}}))$ and $g_1=((g, c))$ are not homotopic in $C(S^1, S^1\times S^1)$, by Lemma \ref{tiam}. But $\phi(f_1)=P_1f_1=f\simeq g=P_1g_1=\phi(g_1)$. 
\end{exam}
Let $*_1$ and $*_2$ be two operations on two groups $C(X, Y)$ and $C(X, Z)$, respectively. Then, for $f_1, g_1\in C(X, Y)$ and $f_2, g_2\in C(X, Z)$, we have the operation $*$ on the product group $\pi_1[C(X, Y)]\times\pi_1[C(X, Z)]$ as the following.
\begin{equation*}
([f_1], [g_1])* ([f_2], [g_2])=([f_1]*_1^{\prime} [f_2], [g_1]*_2^{\prime} [g_2])=([f_1 *_1 f_2], [g_1 *_2 g_2]).
\end{equation*}
\begin{thm}
Let $X$ be a topological space and $Y, Z$ be two topological groups. Then \[\pi_1[C(X, Y\times Z)]\cong\pi_1[C(X, Y)]\times\pi_1[C(X, Z)].\]
\end{thm}
\begin{proof}
Define $\phi: \pi_1[C(X, Y\times Z)]\longrightarrow \pi_1[C(X, Y)]\times\pi_1[C(X, Z)]$ by $\phi([f])=([P_1f], [P_2f])$. We claim that $\phi$ is a group isomorphism. First, assume $[f], [g]\in C(X, Y\times Z)$ and $[f]=[g]$. Then $f\simeq g$. This implies $P_1f\simeq P_1g$ and $P_2f\simeq P_2g$. Thus $[P_1f]=[P_1g]$ and $[P_2f]=[P_2g]$. Hence, $\phi([f])=([P_1f], [P_2f])=([P_1g], [P_2g])=\phi([g])$. Let  $f, g \in C(X, Y\times Z)$ and $\bullet$ be the operation on $C(X, Y\times Z)$.  Then 
\begin{align*}
\phi([f]\bullet^{\prime}[g])=\phi([f\bullet g])=([P_1(f\bullet g)], [P_2(f\bullet g)])=
([P_1f*_1 P_1g], [P_2f *_2 P_2g])=
\\
([P_1f]*_1^{\prime} [P_1g], [P_2f]*_2^{\prime} [P_2g])=([P_1f], [P_2f])* ([P_1g], [P_2g])=\phi([f])*\phi([g]).
\end{align*}
Thus $\phi$ is a group homomorphism. Now let $f, g\in C(X, Y\times Z)$. It is easy to see that $f=((P_1f, P_2f))$ and $g=((P_1g, P_2g))$. Let $\phi([f])=\phi([g])$. Then $([P_1f], [P_2f])=([P_1g], [P_2g])$. This implies $P_1f\simeq P_1g$ and $P_2f\simeq P_2g$. By Lemma \ref{tiam}, $f\simeq g$, i.e., $[f]=[g]$. Thus $\phi$ is one-one. Finally, to show $\phi$ is an onto map, consider $([f],[g])\in \pi_1[C(X,Y)]\times\pi_1[C(X, Z)]$. Then we have $((f, g))\in C(X, Y\times Z)$ and 
$\phi([((f, g))])=([P_1((f, g))], [P_2((f, g))])=([f], [g])$.
\end{proof}
For a topological space $X$, the two groups $C(X, S^1\times R)$ and $C(X, S^1)$ may not be isomorphic. However, we can utilize the previous theorem to obtain the following result.
\begin{equation*}
\pi_1[C(X, S^{1}\times\Bbb R)]\cong \pi_1[C(X, S^{1})]\times\pi_1[C(X, \Bbb R)]\cong\pi_1[C(X, S^{1})]\times 0\cong\pi_1[C(X, S^{1})].
\end{equation*}

\section{Ring homotopy}
In this section, we consider a ring of continuous maps of the form $C(X, Y)$, where $X$ is a topological space and $Y$ is a topological ring. In this case, the defined product and sum of elements in $Y$ are inherited by $C(X, Y)$. Therefore, if $f, g \in C(X, Y)$, then both $f\cdot g$ and $f+g$ are elements of $C(X, Y)$.

Moreover, since both the addition ($+$) and multiplication ($\cdot$) operations are continuous on $Y\times Y$ as maps, the map $h: X\longrightarrow Y\times Y$ defined by $h(x) = (f(x), g(x))$ is also continuous. This allows us to conclude that $f\cdot g = \cdot\circ h$ and $f+g = +\circ h$ are continuous maps from $X$ to $Y$, as the composition of continuous maps is continuous.

The next example shows that the topological ring property for $Y$ is not superfluous. 
\begin{exam}
Let $(X, \tau)$ be a non-discrete topological space. Put $Y=P(X)$ (i.e., the power set of $X$). Consider the collection $\mathcal{B}=\{U(B): B\in\tau\}$, where $U(B)=\{A\subset X: A\subset B\}$. Then it is easy to see that $\mathcal{B}$ is a base for a topology on $Y$. Since, for each two $B_1, B_2\in \mathcal{B}$, if $A\in U(B_1)\cap U(B_2)$, then $A\in U(B_1\cap B_2)\subseteq U(B_1)\cap U(B_2)$, and $Y=\bigcup_{\alpha\in S, U_{\alpha}\in\tau}U(U_{\alpha})$, where $X=\bigcup_{\alpha\in S}U_{\alpha}$. So $Y$ is a topological space. On the other hand, $Y$ together with the addition $+=\Delta$ and multiplication $.=\cap$ is a commutative ring. However, $C(X, Y)$ is not a ring. Consider the map $f: X\longrightarrow Y$ by $f(x)= \{x\}$. For each open base element $U(B)$, where $B\in\tau$, we have $f^{-1}(U(B))=B$ which is open in $X$. Thus $f\in C(X, Y)$. Let $C$ be an open set in $X$ and $y\in X\setminus C$ be a non-isolated point. Define $g: X\longrightarrow Y$ by $g(x)=\{y\}$ for each $x\in X$. Then $g\in C(X, Y)$ too.  Thus $(f+g)(x)=(\{x\}\setminus \{y\})\cup (\{y\}\setminus \{x\})$. We have $U(C)$ is an open set in $Y$ and $(f+g)^{-1}(U(C))=\{y\}$, which is not open in $X$. This shows that $f+g\not\in C(X, Y)$.
\end{exam}
\begin{defn}
Let $(A, +, .), (B, +, .)$ be two unitary commutative rings of continuous maps. A bi-homotopy $\phi: A\longrightarrow B$ is called a ring-homotopy if it satisfies the following conditions:
\begin{enumerate}
\item For each $f_1, f_2\in A$, $\phi(f_1+ f_2)\simeq \phi(f_1)+\phi(f_2)$.

\item For each $f_1, f_2\in A$, $\phi(f_1.f_2)\simeq \phi(f_1).\phi(f_2)$.
\end{enumerate}
If there exist two ring-homotopies $\phi: A\longrightarrow B$ and $\psi: B\longrightarrow A$ such that $\phi\psi(h)\simeq h$ for each $h\in B$ and $\psi\phi(h)\simeq h$ for each $h\in A$, respectively, then we say $A$ and $B$ are  ring-homotopic (we denote it by $A\simeq^{r}B$).
\end{defn}
Trivially, if $\phi: A\longrightarrow B$ is a bi-homotopy which also is a ring-homomorphism, then it is a ring-homotopy.

Let us call  $A\subseteq C(X, Y)$ a hom-closed subset if $f\simeq g$ and $g\in A$, then $f\in A$. Evidently, for each $f\in C(X, Y)$, $[f]$ is a hom-closed subset of $C(X, Y)$. Also, the intersection (resp., union) of two hom-closed sets is a hom-closed set. To see an example of a hom-closed ideal in the ring $A$, put $I=\{f\in A: f\simeq 0\}$. Then, it is easy to see that $I$ is a hom-closed ideal of $A$ and it is  a proper ideal  of $A$ \ifif 1 is not homotopic to the zero in $A$. It also is easy to see that $I$ is the smallest hom-closed ideal of $A$.

This is important to mention that if $f_1, g_1$ and $f_2, g_{2}$  are elements of $C(X, Y)$, $f_1\simeq g_1$ and $f_2\simeq g_2$,  then $f_1.f_2\simeq g_1.g_2$ (resp., $f_1+f_2\simeq g_1+g_2$). For, let $F: X\times I\longrightarrow Y$ be a homotopy between $f_1$ and $g_1$ and $G: X\times I\longrightarrow Y$ be a homotopy between $f_2$ and $g_2$. Define $H: X\times I\longrightarrow Y$ by $H(x, t)=F(x, t). G(x, t)$ (resp., $H(x, t)=F(x, t)+G(x, t)$). Then,  trivially $H$ is a homotopy between $f_1.f_2$ (resp., $f_1+f_2$) and $g_1.g_2$ (resp., $g_1+g_2$).
\begin{lem}\label{mary}
The following statements hold.
\begin{enumerate}
\item If there exists a ring-homotopy $\phi$ from $A$ to $B$ and $I$ is an ideal of $A$, then the set $I_{\phi}=\{g\in B: g\simeq \phi(f)$, for some $f\in I\}$ is an ideal of $B$.

\item Let $A\simeq^{r}B$ (i.e., there are two maps $\phi$ and $\psi$ satisfy in the definition $A\simeq^{r} B$) and $I$ be a maximal ideal of $A$. Then $I$ is hom-closed  \ifif $I_{\phi}$ is a maximal ideal of $B$.
\end{enumerate}
\end{lem}
\begin{proof}
(1) Consider two elements $f, g\in I_{\phi}$. Then there are two elements $f_1, g_1\in I$ such that $f\simeq\phi(f_1)$ and $g\simeq\phi(g_1)$. Thus we have \[f+g\simeq\phi(f_1)+\phi(g_1)\simeq\phi(g_1+f_1),\quad \text{where}\quad f_1+g_1\in I.\] This implies $f+g\in I_{\phi}$. Now, let $f\in B$ and $h\in I_{\phi}$. Then there exist $f_1\in A$ and $h_1\in I$ such that $\phi(f_1)\simeq f$ and $\phi(h_1)\simeq h$. Thus $f.h\simeq \phi(f_1).\phi(h_1)\simeq \phi(f_1h_1)$, where $f_1h_1\in I$. This shows $f.h\in I_{\phi}$. Hence  $I_{\phi}$ is an ideal of $B$.

(2) $\Longrightarrow$ First, we claim that $\phi(1_A)\simeq 1_B$ and $\phi(0_A)\simeq 0_B$. Since, we have $\phi(0)=\phi(0+0)\simeq \phi(0)+ \phi(0)$. Thus $\phi(0)\simeq 0$. And, $\phi(1)=\phi(1.1)\simeq \phi(1). \phi(1)$. This implies $\phi(1)(1-\phi(1))\simeq 0$. Thus $\psi(\phi(1)(1-\phi(1)))\simeq \psi(0)\simeq 0$. But the left side is homotopic to the $\psi(1)-1$. Hence $\psi(1)\simeq 1$. This shows that $1\simeq \phi(\psi(1))\simeq\phi(1)$. By Part (1), $I_{\phi}$ is an ideal of $B$. Now let $f\in B$ and $f\not\in I_{\phi}$. Then there exists $g\in A$ such that $\phi(g)\simeq f$. $f\not\in I_{\phi}$ implies $g\not\in I$. By maximality of $I$, there exists $h\in A$ such that $1-gh\in I$. Thus we have \[\phi(1-gh)\simeq 1-\phi(g).\phi(h)\simeq 1-f.\phi(h).\] Thus $1-f.\phi(h)\in I_{\phi}$, where $\phi(h)\in B$. This shows $I_{\phi}$ is a maximal ideal of $B$.

$\Leftarrow$ Let $f\simeq g$ and $f\in I$. If $g\not\in I$, then there exists $h\in A$ such that $1-gh\in I$, by maximality of $I$. This implies that;\[\phi(1-gh)\simeq 1-\phi(g).\phi(h)\simeq 1-\phi(f).\phi(h).\] Thus $1-\phi(f). \phi(h)\in I_{\phi}$. On the other hand, $\phi(f)\in I_{\phi}$. Hence $1\in I_{\phi}$, a contradiction. So we must have $g\in I$.
\end{proof}

The set of all maximal ideals of a ring $A$ with the Zariski topology is a well-known topological space denoted by $\Max(A)$. We denote by $\Max_H(A)$ the subspace of $\Max(A)$ consisting of hom-closed maximal ideals of $A$.
\begin{thm}\label{had}
If $A\simeq^{r}B$, then two spaces $\Max_H(A)$ and $\Max_H(B)$ are homeomorphic.
\end{thm}
\begin{proof} Let $\phi$ and $\psi$ be maps in the definition $A\simeq^{r}B$. Define maps  $f: \Max_H(A)\longrightarrow\Max_H(B)$ and $g: \Max_H(B)\longrightarrow\Max_H(A)$ by $f(M)=M_{\phi}$ and $g(M)=M_\psi$, respectively. First, we show $f$ is well-defined, for $g$ is similar to it. By Lemma \ref{mary}, for each $M\in\Max_H(A)$, we have $f(M)\in\Max_H(B)$. Let $M=N\in \Max_H(A)$ and $f\in M_{\phi}$. Then $f\simeq \phi(g)$ for some $g\in M=N$. This implies $f\in N_{\phi}$, i.e., $M_\phi\subseteq N_\phi$. Similarly, we can see that $N_\phi\subseteq M_\phi$. Thus $\phi(M)=\phi(N)$. Moreover, for each $M\in\Max_H(B)$ we have, \[(f\circ g)(M)=f(g(M))=f(M_\psi)=(M_\psi)_{\phi}=\{g\in B: g\simeq\phi(h), h\in M_\psi\}=\]\[\{g\in B: g\simeq\phi(h), h\simeq\psi(h_1), h_1\in M\}=\{g\in B: g\simeq \phi(\psi(h_1))\simeq h_1, h_1\in M\}=M.\]
Similarly, we can see that for each $M\in\Max_H(A)$, $(g\circ f)(M)=M$. It remains to show that $f$ and $g$ are continuous. We show it for $f$, for $g$ is similar to that. Consider an open base element $D(g)$ in $\Max_H(B)$, where $g\in B$. There exists $h\in A$ such that $\phi(h)\simeq g$. We claim that $f^{-1}(D(g))=D(h)$. Hence this is open. To prove our claim, first, we assume $M\in f^{-1}(D(g))$ and $h\in M$.  Then $f(M)=M_\phi\in D(g)$ and $\phi(h)\in M_\phi$. This implies $g\in M_\phi$, since $M_\phi$ is hom-closed. That is $g\in f(M)$, a contradiction. This shows $f^{-1}(D(g))\subseteq D(h)$. To see another inclusion, let $M\in D(h)$ and $g\in M_\phi$ (i.e., $M\not\in f^{-1}(D(g))$). Then $\phi(h)\in M_\phi$. Thus $\phi(h)\simeq\phi(k)$, for some $k\in M$. Thus $h\simeq K$ and hence $h\in M$. This is a contradiction. So we are done.
\end{proof}
\begin{exam}
Let $X$ and $Y$  be two topological spaces and $Z$ be a topological ring. If $X, Y$ are two equivalent spaces and $A=C(X, Z), B=C(Y, Z)$, then $A\simeq B$, by Proposition \ref{ali}. Now, it is easy to see that the defined bi-homotopies between these two rings (i.e., $\psi$ and $\phi$ in the proof of Proposition \ref{ali}) are ring-homotopies, hence $A\simeq^r B$. Thus, by Theorem  \ref{had}, we have $\Max_H(A)$ and $\Max_H(B)$ are two homeomorphism spaces. 
\end{exam}
If $(A, +, .)$ is a ring of continuous maps, then it is easy to see that $\pi_1[A]$ with two operations $[f]+^{\prime} [g]=[f+g]$ and $[f].^{\prime} [g]=[f.g]$  is a ring. Let $I$ be an ideal of $A$.  Put
\begin{equation*}
I_h={f\in A: f\simeq g \quad \text{for some} \quad g\in I}.
\end{equation*}
Then, it is easy to see that $I_h$ is a hom-closed ideal of $A$ and $I\subseteq I_h$. If $J$ is a hom-closed ideal of $A$ containing $I$, and $f\in I_h$, then $f\simeq g$ for some $g\in I$. Thus, $g\in J$, and hence $f\in J$. This shows that $I_h$ is the smallest hom-closed ideal containing $I$.
The following result is easy to prove.
\begin{prop}
Let $I$ and $J$ be two ideals of the ring $A$.
\begin{enumerate}
\item If $I\subseteq J$, then $I_h\subseteq J_h$.
\item $(I+J)_h=I_h+J_h$.
\item $I$ is a hom-closed ideal of $A$ if and only if $I=I_h$.
\item The sum of two hom-closed ideals in $A$ is a hom-closed ideal in $A$.
\end{enumerate}
\end{prop}
Define the map $\phi: A\longrightarrow\pi_1[A]$ by $\phi(f)=[f]$. It is easy to see that $\phi$ is an epimorphism (surjective ring-homomorphism). Thus, $A/\text{ker}\ \phi\cong\pi_1[A]$. This shows that $\pi_1[A]$ is a factor ring of $A$.

\begin{exam}
(1) Let $X$ be a completely regular Hausdorff space, and $\Bbb R$ be the real numbers with the standard topology. It is well-known that $C(X, \Bbb R)$ is a ring (see \cite{GJ}), and each $f\in C(X, \Bbb R)$ is homotopic to zero. Thus, we have $\pi_1[C(X, \Bbb R)]={[0]}$. On the other hand, since $1$ is homotopic to the zero in this ring, the only hom-closed ideal is the ring itself.

(2) Let $Y$ be a topological ring that is not a contractible space (e.g., $\Theta W_{\text{Ring}}(\Bbb Z/(2))$, see Example 11 in \cite{R} and \cite{T}), and let $X$ be any space. Then there exists an $f\in C(X, Y)$ that is not homotopic to the constant function $0$. This implies $\pi_1[C(X, Y)]$ is a non-zero ring.
\end{exam}
Now, we want to present the form of ideals (resp., prime ideals and maximal ideals) of ring $\pi_1[A]$.

\begin{prop}
Let $A$ be a ring of continuous maps.
\begin{enumerate}
\item Every ideal of $\pi_1[A]$ is of the form $\pi_1[I]$, where $I$ is a hom-closed ideal of $A$.
\item Every prime ideal of $\pi_1[A]$ is of the form $\pi_1[P]$, where $P$ is a hom-closed prime ideal of $A$.
\item Every maximal ideal of $\pi_1[A]$ is of the form $\pi_1[M]$, where $M$ is a hom-closed maximal ideal of $A$.
\end{enumerate}
\end{prop}

\begin{proof}
(1) Let $J$ be an ideal of $\pi_1[A]$. Put $I=\{f\in A: [f]\in J\}$. We claim that $I$ is a hom-closed ideal of $A$, and $\pi_1[I]=J$.

Trivially, $I$ is a hom-closed subset of $A$. Let $f, g\in I$. Then $[f], [g]\in J$, and since $J$ is an ideal, $[f] +^{\prime} [g] = [f+g] \in J$. This implies $f+g \in I$. Now, consider $g\in A$ and $f\in I$. Then $[g]\in \pi_1[A]$ and $[f]\in J$. Thus, $[f.g] = [f] .^{\prime} [g] \in J$. Therefore, $fg \in I$. Hence, we have shown that $I$ is an ideal of $A$.
It is easy to see that $\pi_1[I]$ is an ideal of $\pi_1[A]$ and $J\subseteq \pi_1[I]$. Let $[f]\in \pi_1[I]$. Then $f\in I$ and hence $[f]\in J$. This shows that $\pi_1[I] \subseteq J$. Thus, $\pi_1[I] = J$.

(2) Let $Q$ be a prime ideal of $\pi_1[A]$. By Part 1, $Q = \pi_1[P]$ for some hom-closed ideal $P$ of $A$. To show that $P$ is a prime ideal, consider $f, g\in A$ with $f.g \in P$. Then $[f] .^{\prime} [g] = [f.g] \in \pi_1[P] = Q$. Thus $[f] \in Q$ or $[g] \in Q$, since $Q$ is a prime ideal. This shows that $f\in P$ or $g\in P$, since $P$ is hom-closed. Therefore, $P$ is a prime ideal of $A$.

(3) Let $N$ be a maximal ideal of $\pi_1[A]$. By Part 1, $N = \pi_1[M]$ for some hom-closed ideal $M$ of $A$. To show that $M$ is a maximal ideal of $A$, consider $f\not\in M$. Then $[f] \not\in \pi_1[M] = N$, since $M$ is hom-closed. By the maximality of $N$, there is $g\in M$ such that $[1-fg] = 1 - [f.g] = [1] +^{\prime} (-[f]) .^{\prime} [g] \in N = \pi_1[M]$. This implies $1 - fg \in M$, i.e., $M$ is a maximal ideal of $A$.
\end{proof}


\begin{thebibliography}{10}
\bibitem{AM} M. F. Atiyah, \textit{Introduction to commutaitive algebra}, University of Oxfoard, 1969.
\bibitem{GJ}
{L. Gillman and  M. Jerison, \textit{Rings of Continuous Maps},
Springer, $1976$.}
\bibitem{H}A. Hatcher, \textit{Algebraic topology}, Cambridge
University Press, Cambridge, $2002$.
\bibitem{M}
J. P. May, \textit{A Concise Course in Algebraic Topology}, University of Chicago Press, 1999.
\bibitem{R}D. M. Robert, \textit{The universal simplicial bundle is a
simplicial group}, New York J. Math. 19 (2013) 51–60.
\bibitem{R_1}J. J. Rotman, \textit{Advanced modern algebra}, Prentice Hall; 1st edition (2002); 2nd printing (2003).
\bibitem{T}Trimble, Todd. (mathoverflow.net/users/2926), Noncontractible connected topological rings? http://mathoverflow.net/questions/119962 (version: 2013-01-26).
\end{thebibliography}
\end{document}

$\Leftarrow$ Let $\theta$ be a group isomorphism from $\pi_{1}(X, x_{0})$ onto $\pi_{1}(Y, y_{0})$. Put $\phi: L_1(X, x_0)\longrightarrow L_1(Y, y_0)$  by $\phi(f)=g$, where, $\theta([f])=[g]$ and $\psi: L_1(Y, y_0)\longrightarrow L_1(X, x_0)$,  by $\psi(g)=f$, where $\theta^{-1}([g])=[f]$. We claim that $\phi$ and $\psi$ are two bi*-homotopies and we have $\psi\phi(h)\simeq h$ for each $h\in L_1(X, x_0)$ and $\psi\phi(h)\simeq h$ for each $h\in L_1(Y, y_0)$. We prove $\phi$ is a bi*-homotopy, for $\psi$ is similar to that. First assume $f_1, f_2\in L_1(X, x_0)$ and $f_1\simeq f_2$. Then $[f_1]=[f_2]$ and hence $[g_1]=\theta([f_1])=\theta([f_2])=[g_2]$. This implies $\phi(f_1)=g_1\simeq g_2=\phi(f_2)$. Next, let $\phi(f_1)=\phi(f_2)$. Then
$[f_1]=\theta^{-1}\theta([f_1])=\theta^{-1}\theta([f_2])=[f_2]$. Thus $f_1\simeq f_2$. Consider $k\in L_1(Y, y_0)$. Then there exists $f\in L_1(X, x_0)$ such that $\theta([f])=[k]$. Thus $\phi(f)=k$. Now let $f_1, f_2\in L_1(X, x_0)$. Then we have $\phi(f_1)=g_1$ and $\phi(f_2)=g_2$, where $\theta([f_1])=[g_1]$ and $\theta([f_2])=[g_2]$, respectively. Hence  $\theta([f_1 * f_2])=\theta([f_1]o[f_2])=\theta([f_1])o \theta([f_2])=[g_1]o[g_2]=[g_1 * g_2]$. This implies $\phi(f_1 * f_2)=g_1 * g_2=\phi(f_1)*\phi(f_2)$.